\documentclass[12pt,reqno]{amsart}
\usepackage{etoolbox}
\makeatletter

\patchcmd{\@setaddresses}{\scshape\ignorespaces}{\ignorespaces}{}{} 

\appto\maketitle{%
\let\@makefnmark\relax  \let\@thefnmark\relax
\ifx\@empty\addresses\else\@footnotetext{%
  \vskip-\bigskipamount\@setaddresses}
  }
\def\enddoc@text{}
\makeatother

\makeatletter
\patchcmd\maketitle
  {\uppercasenonmath\shorttitle}
  {}
  {}{}
\patchcmd\maketitle
  {\@nx\MakeUppercase{\the\toks@}}
  {\the\toks@}
  {}
  {}{}
\patchcmd\@settitle{\uppercasenonmath\@title}{\Large}{}{}
\patchcmd\@setauthors
  {\MakeUppercase{\authors}}
  {\authors}
  {}{}
\makeatother
\usepackage{amsmath,amssymb,amsthm}
\usepackage{color}
\usepackage{url}
\usepackage{tikz-cd}
\usepackage[utf8]{inputenc}
\usepackage[T1]{fontenc}
\usepackage{geometry}
\newtheorem{theorem}{Theorem}[section]
\newtheorem{definition}{Definition}[section]

\newtheorem{corollary}{Corollary}[section]
\newtheorem{proposition}{Proposition}[section]
\newtheorem{lemma}{Lemma}[section]
\newtheorem{remark}{Remark}[section]
\newtheorem{example}{Example}[section]
\numberwithin{equation}{section}
\usepackage[colorlinks=true]{hyperref}
\hypersetup{urlcolor=blue, citecolor=red , linkcolor= blue}
\begin{document}
\address{$^{[1]}$ University of Sfax, Tunisia.}
\email{\url{kais.feki@hotmail.com}}


\date{September 24, 2019}
\author[Kais Feki] {\Large{Kais Feki}$^{1}$}
\title[Spectral radius of semi-Hilbertian space operators and its applications]{Spectral radius of semi-Hilbertian space operators and its applications}

\begin{abstract}
In this paper, we aim to introduce the notion of the spectral radius of bounded linear operators acting on a complex Hilbert space $\mathcal{H}$, which are bounded with respect to the seminorm induced by a positive operator $A$ on $\mathcal{H}$. Mainly, we show that $r_A(T)\leq \omega_A(T)$ for every $A$-bounded operator $T$, where $r_A(T)$ and $\omega_A(T)$ denote respectively the $A$-spectral radius and the $A$-numerical radius of $T$. This allows to establish that $r_A(T)=\omega_A(T)=\|T\|_A$ for every $A$-normaloid operator $T$, where $\|T\|_A$ is denoted to be the $A$-operator seminorm of $T$. Moreover, some characterizations of $A$-normaloid and $A$-spectraloid operators are given.
\end{abstract}
\maketitle

\section{Introduction and Preliminaries}\label{s1}
Let $\mathcal{H}$ be a non trivial complex Hilbert space with inner product $\langle\cdot\mid\cdot\rangle$ and associated norm $\|\cdot\|$. Let $\mathcal{B}(\mathcal{H})$ denote the algebra of bounded linear operators on $\mathcal{H}$.

In all that follows, by an operator we mean a bounded linear operator. The range of every operator is denoted by $\mathcal{R}(T)$, its null space by $\mathcal{N}(T)$ and $T^*$ is the adjoint of $T$. For the sequel, it is useful to point out the following facts. Let  $\mathcal{B}(\mathcal{H})^+$ be the cone of positive (semi-definite) operators, i.e.
$\mathcal{B}(\mathcal{H})^+=\left\{A\in \mathcal{B}(\mathcal{H})\,;\,\langle Ax\mid x\rangle\geq 0,\;\forall\;x\in \mathcal{H}\;\right\}$. Any $A\in \mathcal{B}(\mathcal{H})^+$ defines a positive semi-definite sesquilinear form as follows:
$$\langle\cdot\mid\cdot\rangle_{A}:\mathcal{H}\times \mathcal{H}\longrightarrow\mathbb{C},\;(x,y)\longmapsto\langle x\mid y\rangle_{A} :=\langle Ax\mid y\rangle.$$
Notice that the induced seminorm is given by $\|x\|_A=\langle x\mid x\rangle_A^{1/2}$, for every $x\in \mathcal{H}$. This makes $\mathcal{H}$ into a semi-Hilbertian space. One can check that $\|\cdot\|_A$ is a norm on $\mathcal{H}$ if and only if $A$ is injective, and that $(\mathcal{H},\|\cdot\|_A)$ is complete if and only if $\mathcal{R}(A)$ is closed. Further, $\langle\cdot\mid\cdot\rangle_{A}$ induces a seminorm on a certain subspace of $\mathcal{B}(\mathcal{H})$ as follows. Given $T\in\mathcal{B}(\mathcal{H})$, if there exists $c>0$ satisfying $\|Tx \|_{A} \leq c \|x \|_{A}$, for all $x\in
\overline{\mathcal{R}(A)}$ it holds:
\begin{equation*}\label{semii}
\|T\|_A:=\sup_{\substack{x\in \overline{\mathcal{R}(A)},\\ x\not=0}}\frac{\|Tx\|_A}{\|x\|_A}=\displaystyle\sup_{\substack{x\in \overline{\mathcal{R}(A)},\\ \|x\|_A= 1}}\|Tx\|_{A}<\infty.
\end{equation*}

From now on, we suppose that $A\neq0$ and we denote $\mathcal{B}^{A}(\mathcal{H}):=\left\{T\in \mathcal{B}(\mathcal{H})\,;\,\|T\|_{A}< \infty\right\}$. It can be seen that $\mathcal{B}^{A}(\mathcal{H})$ is not a subalgebra of $\mathcal{B}(\mathcal{H})$. Moreover, if $T\in\mathcal{B}^{A}(\mathcal{H})$, then one can check that $\|T\|_A=0$ if and only if $ATA=0$.

\begin{definition} (\cite{acg1})
Let $A\in\mathcal{B}(\mathcal{H})^+$ and $T \in \mathcal{B}(\mathcal{H})$. An operator $S\in\mathcal{B}(\mathcal{H})$ is called an $A$-adjoint of $T$ if for every $x,y\in \mathcal{H}$, the identity $\langle Tx\mid y\rangle_A=\langle x\mid Sy\rangle_A$ holds. That is $S$ is solution in $\mathcal{B}(\mathcal{H})$ of the equation $AX=T^*A$.
\end{definition}
The existence of an A-adjoint operator is not guaranteed. The set of all operators which admit $A$-adjoints is denoted by $\mathcal{B}_{A}(\mathcal{H})$. By Douglas Theorem \cite{doug}, we  have
\begin{align*}
\mathcal{B}_{A}(\mathcal{H})
& = \left\{T\in \mathcal{B}(\mathcal{H})\,;\;\mathcal{R}(T^{*}A)\subseteq \mathcal{R}(A)\right\}\\
 &=\left\{T \in \mathcal{B}(\mathcal{H})\,;\;\exists \,\lambda > 0\,;\;\|ATx\| \leq \lambda \|Ax\|,\;\forall\,x\in \mathcal{H}  \right\},
\end{align*}
If $T\in \mathcal{B}_A(\mathcal{H})$, the reduced solution of the equation
$AX=T^*A$ is a distinguished $A$-adjoint operator of $T$, which is denoted by $T^\sharp$. Note that, $T^\sharp=A^\dag T^*A$ in which $A^\dag$ is the Moore-Penrose inverse of $A$. For more results concerning $T^\sharp$ see \cite{acg1,acg2}.

Again, by applying Douglas theorem, it can observed that
$$\mathcal{B}_{A^{1/2}}(\mathcal{H})=\left\{T \in \mathcal{B}(\mathcal{H})\,;\;\exists \,\lambda > 0\,;\;\|Tx\|_{A} \leq \lambda \|x\|_{A},\;\forall\,x\in \mathcal{H}  \right\}.$$
Operators in $\mathcal{B}_{A^{1/2}}(\mathcal{H})$ are called $A$-bounded. Notice that, if $T\in \mathcal{B}_{A^{1/2}}(\mathcal{H})$, then $T(\mathcal{N}(A))\subseteq \mathcal{N}(A)$. Moreover, it was proved in \cite{fg} that if $T\in\mathcal{B}_{A^{1/2}}(\mathcal{H})$, then
\begin{align}\label{aseminorm}
\|T\|_A
&=\sup\left\{\|Tx\|_{A}\,;\;x\in \mathcal{H},\,\|x\|_{A}= 1\right\}\\
&=\sup\left\{|\langle Tx\mid y\rangle_A|\,;\;x,y\in \mathcal{H},\,\|x\|_{A}=\|y\|_{A}= 1\right\}\nonumber.
\end{align}
For the sequel, it is important to point out the fact that for every $T,S\in \mathcal{B}_{A^{1/2}}(\mathcal{H})$ we have
\begin{equation}\label{crucial0}
\|TS\|_A\leq \|T\|_A\cdot\|S\|_A.
\end{equation}
Also, if $T\in \mathcal{B}_{A^{1/2}}(\mathcal{H})$, then $\|T\|_A=0$ if and only if $AT=0$. For more details concerning $A$-bounded operators, see \cite{acg3} and the references therein.

Note that $\mathcal{B}_{A}(\mathcal{H})$ and $\mathcal{B}_{A^{1/2}}(\mathcal{H})$ are two subalgebras of $\mathcal{B}(\mathcal{H})$ which are neither closed nor dense in $\mathcal{B}(\mathcal{H})$. Moreover, the following inclusions $\mathcal{B}_{A}(\mathcal{H})\subseteq\mathcal{B}_{A^{1/2}}(\mathcal{H})\subseteq
\mathcal{B}^{A}(\mathcal{H})\subseteq \mathcal{B}(\mathcal{H})$ hold with equality if $A$ is injective and has closed range. For the sake of completeness, we will give examples that show that the above inclusions are in general strict. For an account of results, we refer to \cite{acg1,acg2,acg2} and the references therein.

It is useful to recall that an operator $T$ is called $A$-self-adjoint if $AT$ is self-adjoint (that is $AT=T^*A$) and it is called $A$-positive if $AT\geq0$ and we write $T\geq_{A}0$. Observe that if $T$ is A-selfadjoint, then $T\in\mathcal{B}_A(\mathcal{H})$. However it does not mean, in general, that $T=T^\sharp$. Furthermore, an operator $T\in\mathcal{B}_A(\mathcal{H})$ satisfies $T=T^{\sharp}$ if and only if $T$ is $A$-selfadjoint and $\mathcal{R}(T) \subset \overline{\mathcal{R}(A)}$ (see \cite[Section 2]{acg1}).

Recently, there are many papers that study operators defined on semi-Hilbertian spaces. One may see \cite{bakfeki01,bakfeki04,zamani1,majsecesuci} and their references.

Notice also that in \cite{saddi}, Saddi generalized the concepts of numerical radius and spectral radius of an operator and defined the $A$-numerical radius and $A$-spectral radius of an operator $T \in \mathcal{B}(\mathcal{H})$ by
\begin{equation*}\label{numsaddi}
\omega_A(T) = \sup\left\{|\langle Tx\mid x\rangle_A|\,;\;x\in \mathcal{H},\;\|x\|_A= 1\right\}
 \end{equation*}
and
\begin{equation}\label{formulasaddi}
 r_A(T) =\displaystyle\limsup_{n\to \infty}\left\|T^n\right\|_A^{\frac{1}{n}},
 \end{equation}
 respectively. As we will see later, if $T \in \mathcal{B}(\mathcal{H})$ (even if $T \in \mathcal{B}^A(\mathcal{H})$), the the above definition of $r_A$ does not guarantee that $r_A(T)<\infty$. So, one main target of this paper is to provide an alternative definition of $r_A$ that coincides with the formula \eqref{formulasaddi} when $T\in\mathcal{B}_{A^{1/2}}(\mathcal{H})$. Further, the new definition generalizes the well known Gelfand's formula \cite[Proposition 6.21.]{kub002}.

The study of the spectral radius and the joint spectral radius of operators has been the subject of intense research during last decades. For more details, the interested reader is referred to \cite{bakfeki02,bakfeki04} and the references therein. Recently, the concept of of $A$-joint spectral radius associated with a $d$-tuple of operators $\mathbf{T}=(T_1,\cdots,T_d) \in \mathcal{B}_A(\mathcal{H})^d$ (not necessarily to be commuting) was introduced by H.Baklouti et al. in \cite{bakfeki04} as follows.
\begin{definition}(\cite{bakfeki04})
Let $\mathbf{T}=(T_1,\cdots,T_d) \in \mathcal{B}_A(\mathcal{H})^d$ be a $d$-tuple of operators. The $A$-joint spectral radius of $\mathbf{T}$ is given by
$$r_A(\mathbf{T})= \inf_{n\in\mathbb{N}^*}\left(\bigg\|\sum_{g\in \mathbf{G}(n,d)} \mathbf{T}_g^\sharp \mathbf{T}_{g}\bigg\|_A^{\frac{1}{2n}} \right),$$
where $\mathbb{N}^*$ is the set of positive natural numbers. Also, $\mathbf{G}(n,d)$ denotes the set of all functions from $\{1,\cdots,n\}$ into  $\{1,\cdots,d\}$ and $\mathbf{T}_g:=\prod_{k=1}^nT_{g(k)}$, for $g\in \mathbf{G}(n,d)$.
\end{definition}
Moreover, the following theorem was proved in \cite{bakfeki04}.
\begin{theorem}
Let $\mathbf{T}=(T_1,\cdots,T_d) \in \mathcal{B}_A(\mathcal{H})^d$ be a $d$-tuple of operators. Then,
\begin{equation}\label{jointArad}
r_A(\mathbf{T})=\lim_{n\to+\infty}\left(\bigg\|\sum_{g\in \mathbf{G}(n,d)} \mathbf{T}_g^\sharp \mathbf{T}_{g}\bigg\|_A^{\frac{1}{2n}} \right).
\end{equation}
\end{theorem}
Observe from \eqref{jointArad} that if $d=1$, then the $A$-spectral radius of an operators $T\in \mathcal{B}_A(\mathcal{H})$ is given by
\begin{equation}\label{9dim}
r_A(T)=\inf_{n\in \mathbb{N}^*}\|T^n\|_A^{\frac{1}{n}}=\displaystyle\lim_{n\to\infty}\|T^n\|_A^{\frac{1}{n}}.
\end{equation}
As we will see later, \eqref{9dim} remains also true for every $T\in \mathcal{B}_{A^{1/2}}(\mathcal{H})$.

It is known that for $T\in\mathcal{B}(\mathcal{H})$, one has
$$r(T)\leq \omega(T)\leq \|T\|,$$
 where $r(\cdot)$, $\omega(\cdot)$ and $\|\cdot\|$ are the classical spectral, numerical radii and operator norm of $T$, respectively (see \cite{goldtad01}). Further, the above inequalities become equalities if $T$ is normaloid (i.e. satisfies $r(T)=\|T\|$). For more details see \cite{furuta0}.

 A fundamental inequality for the numerical radius of Hilbert space operators is the power inequality, which is proved by Berger in \cite{berger} and says that for $T\in \mathcal{B}(\mathcal{H})$, we have
\begin{equation}\label{spetralclnumerical01}
\omega(T^n)\leq \omega(T)^n,\;\forall\,n\in \mathbb{N}^*.
\end{equation}
 In this article, we will extend \eqref{spetralclnumerical01} to the case of $A$-bounded operators and we will establish several results governing $r_A(\cdot)$, $\omega_A(\cdot)$ and $\|\cdot\|_A$. These results will be a natural extensions of the well known case $A=I.$

Notice that the concepts of normal, hyponormal and self-adjoint operators in the context of operators defined on a semi-Hilbertian space are developed. In this paper, we will also introduce the concept of paranormal operators in semi-Hilbertian spaces and we will analyze the relationship between these operators.
 \section{Main Results}\label{s2}
In this section, we present our main results. Before that, let us emphasize the fact that if $T \in \mathcal{B}(\mathcal{H})$ (even if $T \in \mathcal{B}^A(\mathcal{H})$), Saddi's definition of $r_A$ given in \eqref{formulasaddi} does not guarantee that $r_A(T)<\infty$. For the reader's convenience, we state here an example.
\begin{example}\label{ex01}
Let $\mathcal{H}=\ell_{\mathbb{N}^*}^2(\mathbb{C})$ and consider the operator $A= {\text {diag}}(0, 1, 0, \frac{1}{2!}, 0, \frac{1}{3!}, \cdots) \in \mathcal{B} (\mathcal{H})^+$. Let $T\in \mathcal{B}(\ell_{\mathbb{N}^*}^2(\mathbb{C}))$ be such that $Te_n=e_{n+1}$, where $(e_n)_{n\in \mathbb{N}^*}$ denotes the canonical basis of $\ell_{\mathbb{N}^*}^2(\mathbb{C})$. A short calculation reveals that $ATA= 0$. This implies that $ATx=0$ for all $x\in \overline{\mathcal{R}(A)}$. So, we infer that $T\in \mathcal B^A(\mathcal{H})$. Moreover, $\|T\|_A=0$. On the other hand, it can be checked that
$$\|T^2Ae_{2n}\|_A^2=\frac{1}{(n!)^2(n-1)!}\;\text{ and }\;\|Ae_{2n}\|_A^2= \frac{1}{(n!)^3},\;\,\forall\, n\geq 2.$$
Thus,
$$\|T^2Ae_{2n}\|_A=\sqrt{n}\|Ae_{2n}\|_A,\;\,\forall\, n\geq 2.$$
This yields that $T^2 \notin \mathcal B^A(\mathcal{H})$ and then $\|T^2\|_A=+\infty$. Therefore, $r_A(T)=+\infty$.
\end{example}

\begin{remark}
In view of Example \ref{ex01}, we see that $\mathcal{B}^{A}(\mathcal{H})$ is not a subalgebra of $\mathcal{B}(\mathcal{H})$. Also, we observe that the inclusions $\mathcal{B}_{A^{1/2}}(\mathcal{H})\subseteq \mathcal{B}^A(\mathcal{H})\subseteq\mathcal{B}(\mathcal{H})$ are in general strict. Indeed $T\in \mathcal{B}^A(\mathcal{H})$ but one can verify that $T\notin \mathcal{B}_{A^{1/2}}(\mathcal{H})$. Also $T^2\in \mathcal{B}(\mathcal{H})$ but $T^2\notin \mathcal{B}^A(\mathcal{H})$. Furthermore the inclusion $\mathcal{B}_A(\mathcal{H})\subseteq \mathcal{B}_{A^{1/2}}(\mathcal{H})$ is also in general strict. In fact, let $A$ be the diagonal operator on $\ell_{\mathbb{Z}}^2(\mathbb{C})$ given by $Ae_n= A_n e_n$ and $T$ be the operator given by $T e_n=\sqrt{T_n}e_{-n}$ with
$$A_n=\begin{cases}
\frac{1}{n^2}&\text{if}\;\; n\geq 1,\\
0&\text{if}\;\;n=0,\\
\frac{1}{|n|}&\text{if}\;\; n\leq -1.
\end{cases}
\;\text{ and }\;
T_n= \begin{cases}
\frac{1}{n}&\text{if}\;\;n\geq 1,\\
0&\text{else}.
\end{cases}$$
It can be observed that $A$ has not closed range. Moreover, we see that
$$\|Tx\|_A^2=\sum_{n>0} \frac{1}{n^2}\,|x_n|^2\leq\sum_{n>0}\frac{1}{n^2}\,|x_n|^2+\sum_{n<0}\frac{1}{|n|}\,|x_n|^2=\|x\|_A^2,$$
for all $x = \sum_{n \in \mathbb{Z}}x_n e_n \in \ell_{\mathbb{Z}}^2(\mathbb{C})$. Thus, $T\in \mathcal{B}_{A^{1/2}}(\mathcal{H})$. However, we have $\|AT e_n\|^2=\frac{1}{n^3}$ and $\|Ae_n\|^2=\frac{1}{n^4}$ for all $n>0$. This shows that $T\notin\mathcal{B}_{A}(\mathcal{H})$.
\end{remark}

Now, we introduce the following new definition of the spectral radius $r_A$. Henceforth, $A$ is implicitly understood as a positive operator.

\begin{definition}
Let $T\in \mathcal{B}_{A^{1/2}}(\mathcal{H})$. The $A$-spectral radius of $T$ is defined as
$$r_A(T)=\displaystyle\inf_{n\in \mathbb{N}^*}\|T^n\|_A^{\frac{1}{n}}.$$
\end{definition}

In the next theorem, we show the equivalence between our new definition and Saddi's definition \cite{saddi} for $T\in \mathcal{B}_{A^{1/2}}(\mathcal{H})$.
\begin{theorem}\label{thm_spectral}
If $T\in \mathcal{B}_{A^{1/2}}(\mathcal{H})$, then
$$r_A(T)=\displaystyle\lim_{n\to\infty}\|T^n\|_A^{\frac{1}{n}}.$$
\end{theorem}
\begin{proof}
Let $f(n)=\log\|T^n\|_A$. Then, by using \eqref{crucial0} we see that
\begin{align*}
f(m+n)
&=\log\|T^{m+n}\|_A\\
&=\log\|T^m\cdot T^n\|_A\\
&\leq\log\|T^m\|_A+\log\|T^n\|_A\\
&=f(m)+f(n).
\end{align*}
Hence $f$ is  is subadditive. Consequently, by Fekete's lemma \cite{fekete}, $\displaystyle\lim_{n\to\infty}\frac{f(n)}n$ exists and is equal to $\displaystyle\inf_{n\ge1}\frac{f(n)}n$. This completes the proof since the log function is continuous and increasing.
\end{proof}

\begin{remark}
 In Theorem \ref{thm_spectral}, we assumed that $T\in \mathcal{B}_{A^{1/2}}(\mathcal{H})$. In fact this assumption is crucial so that we  have $\|T^m\cdot T^n\|_A\leq \|T^m\|_A\|T^n\|_A$. Notice that this inequality was a key step in proving subadditivity of the function $f$ in Theorem \ref{thm_spectral}. Of course, we observe here that if $T\in \mathcal{B}_{A^{1/2}}(\mathcal{H})$, then so are $T^n,\;n=2,3,\cdots.$
\end{remark}

In the next proposition we collect some properties of the $A$-spectral radius.
\begin{proposition}\label{spectradiprop}
Let $T,S\in \mathcal{B}_{A^{1/2}}(\mathcal{H})$. Then the following assertions hold:
\begin{itemize}
  \item [(1)] If $TS=ST$, then $r_A(TS)\leq r_A(T)r_A(S)$.
  \item [(2)] If $TS=ST$, then $r_A(T+S)\leq r_A(T)+r_A(S)$.
 \item [(3)] $r_A(T^k)=[r_A(T)]^k$ for all $k\in \mathbb{N}^*$.
\end{itemize}
\end{proposition}
\begin{proof}
\noindent (1)\;Notice first that since $TS=ST$, then $(TS)^n=T^nS^n$ for all $n$. So, by using Theorem \ref{thm_spectral} together with \eqref{crucial0}, it follows that
\begin{align*}
r_A(TS)
&=\displaystyle\lim_{n\to\infty}\|(TS)^n\|_A^{\frac{1}{n}}\\
&=\displaystyle\lim_{n\to\infty}\|T^nS^n\|_A^{\frac{1}{n}}\\
& \leq \displaystyle\lim_{n\to\infty}\|T^n\|_A^{\frac{1}{n}}\cdot\|S^n\|_A^{\frac{1}{n}}\\
& =r_A(T)\cdot r_A(S).
\end{align*}
\par \vskip 0.1 cm \noindent (2)\;By using the fact that $TS=ST$ and \eqref{crucial0}, we observe that
\begin{equation}\label{binomfor01}
\|(T + S)^n\|_A=\left\|\sum_{k = 0}^n \binom{n}{k} S^k T^{n - k}\right\|\leq \sum_{k = 0}^n \binom{n}{k} \|S\|_A^k \|T\|_A^{n - k}.
\end{equation}
Let $p$ and $q$ be such that $p> r_A(S)$ and $q> r_A(T)$. Then, there exists an integer $m>0$ such that
$$\|S^n\|_A^{\frac{1}{n}}<p\;\;\text{ and }\;\;\|T^n\|_A^{\frac{1}{n}}<q,\;\;\forall\,n\geq m,$$
On the other hand, by using \eqref{crucial0}, we see that
$$\|S^n\|_A^{\frac{1}{n}}\leq\|S\|_A:=s\;\;\text{ and }\;\|T^n\|_A^{\frac{1}{n}}\leq\|T\|_A:=t,\;\;\forall\,n\in \mathbb{N}^*.$$
It follows from \eqref{binomfor01} that, for $n> 2m$,
\begin{align*}
\|(T + S)^n\|_A
& \leq \sum_{k = 0}^{m-1} \binom{n}{k} s^k q^{n - k}+ \sum_{k = m}^{n-m} \binom{n}{k} p^k q^{n - k}+ \sum_{k =n-m+1}^n \binom{n}{k} p^k t^{n - k}\\
 &=\sum_{k = 0}^{m-1} \binom{n}{k} \left(\tfrac{s}{p}\right)^kp^k q^{n - k}+ \sum_{k = m}^{n-m} \tbinom{n}{k} p^k q^{n - k}+ \sum_{k =n-m+1}^n \binom{n}{k} q^{n - k} p^k \left(\tfrac{t}{q}\right)^{n - k}\\
 &\leq\left[\sum_{k = 0}^n \binom{n}{k}p^k q^{n - k}\right] \underbrace{\left[\max_{0\leq i\leq m-1}\left(\tfrac{s}{p}\right)^i+1+ \max_{0\leq i\leq m-1}\left(\tfrac{t}{q}\right)^i\right]}_{=M}=(p+q)^nM,
 \end{align*}
Since $M$ does not depend on $n$, then we get
$$\displaystyle\lim_{n\to\infty} \|(T+S)^n\|_A^{1/n}\leq \displaystyle\lim_{n\to\infty}(p+q)M^{1/n}=p+q.$$
By letting $p$ and $q$ tend respectively to $r_A(S)$ and $r_A(T)$, we obtain the desired inequality.
\par \vskip 0.1 cm \noindent (3)\;For all $k\in \mathbb{N}^*$, we have
$$r_A(T^k)=\displaystyle\lim_{n\to\infty}\|T^{nk}\|_A^{\frac{1}{n}}=\displaystyle\lim_{n\to\infty}\left[\|T^{nk}\|_A^{\frac{1}{nk}}\right]^k=[r_A(T)]^k.$$
\end{proof}

Now, we turn your attention to the study of the relationship between the $A$-spectral radius and the $A$-numerical radius of $A$-bounded operators. Before that, let us emphasize the fact that $\omega_A(T)$ can be equal to $+\infty$ for an arbitrary $T \in \mathcal{B}(\mathcal{H})$ (even if $T \in \mathcal{B}^A(\mathcal{H})$). Indeed, one can take the operators
$A= \begin{pmatrix}1 & 0 \\ 0 & 0 \end{pmatrix}$ and $T= \begin{pmatrix} 0 & 1 \\ 1 & 0 \end{pmatrix}$ on $\mathbb{C}^2$. It is not difficult to see that $T\in\mathcal{B}^{A}(\mathcal{H})$. Also, a straightforward calculation shows that $\omega_A(T)=+\infty$. More precisely we have the following result.
\begin{theorem}
Let $T \in \mathcal{B}(\mathcal{H})$ be such that $T(\mathcal{N}(A))\nsubseteq\mathcal{N}(A)$. Then, $\omega_A(T)=+\infty$.
\end{theorem}
\begin{proof}
Observe that, $\omega_A(T)=\sup\{|z|\,;\;z\in  W_A(T)\}$ where
$$W_A(T)=\{\langle T x\mid x\rangle_A\,;\;x \in \mathcal{H}\;\;\text{and}\;\,\|x\|_A=1\}.$$
 Since $\mathcal{H}=\mathcal{N}(A)\oplus \overline{\mathcal{R}(A)}$, then $x\in\mathcal{H}$ can be written in a unique way into $x = x_1 + x_2$ with $x_1\in\mathcal{N}(A)$ and $x_2\in \overline{\mathcal{R}(A)}$. Since $A\geq 0$, it follows that $\mathcal{N}(A)=\mathcal{N}(A^{1/2})$ which implies that $\|x\|_A=\|x_2\|_A$. Thus, we get
\begin{equation}\label{wat1}
W_A(T)=\{\langle Tx_1\mid x_2\rangle_A+\langle Tx_2\mid x_2\rangle_A\,;\;x_1 \in \mathcal{N}(A),\;x_2 \in \overline{\mathcal{R}(A)}\;\;\text{with}\;\|x_2\|_A=1\}.
\end{equation}
Since $T(\mathcal{N}(A))\nsubseteq\mathcal{N}(A)$, then there exists $a \in \mathcal{N}(A)$ such that $ATa \neq 0$. This implies that $A^{1/2}ATa\neq 0$. Indeed, assume that $A^{1/2}ATa =0$. Then, $A^2Ta=A^{1/2}(A^{1/2}ATa)=0$. So, $\|ATa\|^2=\langle A^2Ta\mid Ta\rangle=0$ and thus $ATa=0$, which is a contradiction. Set $b = \frac{ATa}{\|ATa\|_A} \in \mathcal{R}(A)$. Clearly $\|b\|_A = 1$. Further, if $\mu\in\mathbb{C}$ then $(\mu a, b) \in \mathcal{N}(A)\times \overline{\mathcal{R}(A)}$ with $\|b\|_A = 1$. Hence, by using \eqref{wat1}, we get
\begin{align*}
W_A(T)
&\supseteq\{\langle T(\mu a)\mid b\rangle_A + \langle Tb\mid b\rangle_A\,;\;\mu\in\mathbb{C}\,\}\\
&=\left\{\mu\frac{\|ATa\|^2}{\|ATa\|_A} + \langle Tb\mid b\rangle_A\,;\;\mu\in\mathbb{C}\,\right\}\\
&=\mathbb{C}.
\end{align*}
 Therefore, $W_A(T)= \mathbb{C}$ and thus $\omega_A(T)=+\infty$.
\end{proof}

\begin{remark}
Note that the fact that $W_A(T)= \mathbb{C}$ in the case $T(\mathcal{N}(A))\not\subset\mathcal{N}(A)$ has recently been proved by H. Baklouti et al. in \cite{bakfeki01}. However, our approach here is different from theirs.
\end{remark}
Notice that $\omega_A(T)<+\infty$ for every $T \in \mathcal{B}_{A^{1/2}}(\mathcal{H})$. More precisely, we have the following result.

\begin{proposition}\label{anradius}(\cite{bakfeki01})
If $T\in \mathcal{B}_{A^{1/2}}(\mathcal{H})$, then
\begin{equation}\label{refine1}
\frac{1}{2} \|T\|_A\leq\omega_A(T) \leq \|T\|_A.
\end{equation}
\end{proposition}

The semi-inner product $\langle\cdot\mid\cdot\rangle_A$ induces an inner product on the quotient space $\mathcal{H}/\mathcal{N}(A)$ defined as
$$[\overline{x},\overline{y}] = \langle Ax\mid y\rangle,$$
for all $\overline{x},\overline{y}\in \mathcal{H}/\mathcal{N}(A)$. Notice that $(\mathcal{H}/\mathcal{N}(A),[\cdot,\cdot])$ is not complete unless $\mathcal{R}(A)$ is not closed. However, a canonical construction due to L. de Branges and J. Rovnyak in \cite{branrov} shows that the completion of $\mathcal{H}/\mathcal{N}(A)$  under the inner product $[\cdot,\cdot]$ is isometrically isomorphic to the Hilbert space $\mathcal{R}(A^{1/2})$
with the inner product
$$(A^{1/2}x,A^{1/2}y):=\langle Px\mid Py\rangle,\;\forall\, x,y \in \mathcal{H},$$ where $P$ denotes the orthogonal projection of $\mathcal{H}$ onto the closure of $\mathcal{R}(A)$.

In the sequel, the Hilbert space $\left(\mathcal{R}(A^{1/2}), (\cdot,\cdot)\right)$ will be denoted by $\mathbf{R}(A^{1/2})$ and we use the symbol $\|\cdot\|_{\mathbf{R}(A^{1/2})}$ to represent the norm induced by the inner product $(\cdot,\cdot)$. The fact that $\mathcal{R}(A)\subset \mathcal{R}(A^{1/2}) $ implies that
\begin{align*}
(Ax,Ay)
&=(A^{1/2}A^{1/2}x,A^{1/2}A^{1/2}y) = \langle PA^{1/2}x\mid PA^{1/2}y\rangle\\
 &=\langle A^{1/2}x\mid A^{1/2}y\rangle=\langle x \mid y\rangle_A.
\end{align*}
This leads to the following useful relation:
\begin{equation}\label{usefuleq01}
\|Ax\|_{\mathbf{R}(A^{1/2})}=\|x\|_A,\;\forall\,x\in \mathcal{H}.
\end{equation}
 The interested reader is referred to \cite{acg3} for more information related to the Hilbert space $\mathbf{R}(A^{1/2})$.

 As in \cite{acg3}, we consider the operator $W_A: \mathcal{H}\to \mathbf{R}(A^{1/2})$ defined by
\begin{equation}
W_Ax=Ax,\;\forall\,x\in \mathcal{H}.
\end{equation}
Now, we adapt from \cite{acg3} the following proposition in our context.
\begin{proposition}\label{prop_arias}
Let $T\in \mathcal{B}(\mathcal{H})$. Then $T\in \mathcal{B}_{A^{1/2}}(\mathcal{H})$ if and only if there exists a unique $\widetilde{T}\in \mathcal{B}(\mathbf{R}(A^{1/2}))$ such that $W_AT =\widetilde{T}W_A$.
\end{proposition}
\begin{remark}
In the following diagram we summarize the relationship between operators in $\mathcal{B}_{A^{1/2}}(\mathcal{H})$ and $\mathcal{B}(\mathbf{R}(A^{1/2}))$.
\[\begin{tikzcd}[sep=huge]
\mathcal{H} \arrow[d,"W_A" swap] \arrow[r,"T"]& \mathcal{H} \arrow[d,"W_A"] \\
\mathbf{R}(A^{1/2}) \arrow[r,"\widetilde{T}"] & \mathbf{R}(A^{1/2})
\end{tikzcd}\]
Notice that by \cite[Proposition 2.2.]{acg3}, we have the following result: Given $\widetilde{T}\in \mathcal{B}(\mathbf{R}(A^{1/2}))$, then there exists $T\in\mathcal{B}(\mathcal{H})$ such that $W_AT =\widetilde{T}W_A$ if and only if $\widetilde{T}\mathcal{R}(A)\subset \mathcal{R}(A)$. In such case, there exists a unique $T\in\mathcal{B}_{A^{1/2}}(\mathcal{H})$ such that $\mathcal{R}(T)\subset \overline{\mathcal{R}(A)}$.
\end{remark}
To prove our main result, we first state two lemmas.
\begin{lemma}\label{density01}
Let $A\in \mathcal{B}(\mathcal{H})^+$. Then, $\mathcal{R}(A)$ is dense in $\mathbf{R}(A^{1/2})$.
\end{lemma}
\begin{proof}
Remark first that, since $\mathcal{N}(A)=\mathcal{N}(A^{1/2})$, then $\overline{\mathcal{R}(A)}=\overline{\mathcal{R}(A^{1/2})}$. Moreover, it can be observed that the map $\overline{\mathcal{R}(A)}\rightarrow \mathcal{R}(A^{1/2}),\;x\mapsto A^{1/2}x$ is bijective. So, one can verify that the following map
\begin{align*}
U_A\colon (\overline{\mathcal{R}(A)},\|\cdot\|) & \rightarrow \mathbf{R}(A^{1/2}):=\left(\mathcal{R}(A^{1/2}), (\cdot,\cdot)\right)\\
x&\mapsto A^{1/2}x,
\end{align*}
is an unitary isomorphism. So, since $\mathcal{R}(A^{1/2})$ is dense in $(\overline{\mathcal{R}(A)},\|\cdot\|)$, then $U_A(\mathcal{R}(A^{1/2}))=\mathcal{R}(A)$ is dense in $\mathbf{R}(A^{1/2})$.
\end{proof}

\begin{lemma}\label{norminraundemi}
If $T\in \mathcal{B}(\mathbf{R}(A^{1/2}))$, then
\begin{equation}\label{norminaa1demi}
\|T\|_{\mathcal{B}(\mathbf{R}(A^{1/2}))}=\sup\left\{\|TAx\|_{\mathbf{R}(A^{1/2})}\,;\;x\in \overline{\mathcal{R}(A)},\,\|Ax\|_{\mathbf{R}(A^{1/2})}=1\right\}.
\end{equation}
\end{lemma}

\begin{proof}
By using the fact that $\mathcal{H}=\mathcal{N}(A^{1/2})\oplus \overline{\mathcal{R}(A^{1/2})}$ and the definition of the operator norm, we obtain
\begin{align*}
\|T\|_{\mathcal{B}(\mathbf{R}(A^{1/2}))}
&=\sup\left\{\|Ty\|_{\mathbf{R}(A^{1/2})}\,;\;y\in \mathcal{R}(A^{1/2}),\,\|y\|_{\mathbf{R}(A^{1/2})}=1\right\}\\
&=\sup\left\{\|TA^{1/2}x\|_{\mathbf{R}(A^{1/2})}\,;\;x\in \mathcal{H},\,\|A^{1/2}x\|_{\mathbf{R}(A^{1/2})}=1\right\}\\
&=\sup\left\{\|TA^{1/2}x\|_{\mathbf{R}(A^{1/2})}\,;\;x\in \overline{\mathcal{R}(A^{1/2})},\,\|A^{1/2}x\|_{\mathbf{R}(A^{1/2})}=1\right\}.
\end{align*}
Now, consider the following set
$$\Omega=\left\{\|TAx\|_{\mathbf{R}(A^{1/2})}\,;\;x\in \overline{\mathcal{R}(A)},\,\|Ax\|_{\mathbf{R}(A^{1/2})}=1\right\}.$$
We shall prove that $\|T\|_{\mathcal{B}(\mathbf{R}(A^{1/2}))}=\sup \Omega$. Clearly, we have $\sup \Omega\leq \|T\|_{\mathcal{B}(\mathbf{R}(A^{1/2}))}$. Moreover, let
$$y\in \left\{\|TA^{1/2}x\|_{\mathbf{R}(A^{1/2})}\,;\;x\in \overline{\mathcal{R}(A^{1/2})},\,\|A^{1/2}x\|_{\mathbf{R}(A^{1/2})}=1\right\},$$
then there exists $x\in \overline{\mathcal{R}(A^{1/2})}$ such that $\|A^{1/2}x\|_{\mathbf{R}(A^{1/2})}=1$ and $y=\|TA^{1/2}x\|_{\mathbf{R}(A^{1/2})}$. From Lemma \ref{density01}, there exists a sequence $(x_n)_n\subset \mathcal{H}$ (which can chosen to be in $\overline{\mathcal{R}(A^{1/2})}$ because of the decomposition $\mathcal{H}=\mathcal{N}(A^{1/2})\oplus \overline{\mathcal{R}(A^{1/2})}$) such that $\displaystyle\lim_{n\to +\infty}\|Ax_n-A^{1/2}x\|_{\mathbf{R}(A^{1/2})}=0$. Hence, $y\in \overline{\Omega}$ which yields that $\sup \overline{\Omega}\geq \|T\|_{\mathcal{B}(\mathbf{R}(A^{1/2}))}$. Finally, since $\sup \overline{\Omega}=\sup \Omega$, then
$$\sup \Omega\geq \|T\|_{\mathcal{B}(\mathbf{R}(A^{1/2}))}.$$
 This shows the desired equality \eqref{norminaa1demi}.
\end{proof}

Now we are in a position to prove one of our main results in this paper.
\begin{theorem}\label{thm_spec_numer}
If $T\in \mathcal{B}_{A^{1/2}}(\mathcal{H})$, then
\begin{equation*}
r_A(T)\leq \omega_A(T).
\end{equation*}
\end{theorem}

\begin{proof}
Observe that since $T\in\mathcal{B}_{A^{1/2}}(\mathcal{H})$, then  $T^n\in\mathcal{B}_{A^{1/2}}(\mathcal{H})$ for all $n\in \mathbb{N}^*$. This implies that $T^n(\mathcal{N}(A))\subset \mathcal{N}(A)$ for all $n\in \mathbb{N}^*$. Thus, by using the decomposition $\mathcal{H}=\mathcal{N}(A)\oplus \overline{\mathcal{R}(A)}$ and \eqref{aseminorm}, we see that
$$\|T^n\|_A = \sup\left\{\|T^nx\|_A\,;\;x\in \overline{\mathcal{R}(A)},\,\|x\|_A=1\right\}.$$
Moreover, since $T^n\in\mathcal{B}_{A^{1/2}}(\mathcal{H})$ for all $n\in \mathbb{N}^*$, then by Proposition \ref{prop_arias}, there exists $\widetilde{T^n}\in \mathcal{B}(\mathbf{R}(A^{1/2}))$ such that $W_AT^n =\widetilde{T^n}W_A$. So, by using the equality \eqref{usefuleq01} and Lemma \ref{norminraundemi}, we infer that
\begin{align}\label{compar01}
\|T^n\|_A
& =\sup\left\{\|T^nx\|_A\,;\;x\in \overline{\mathcal{R}(A)},\,\|x\|_A=1\right\}\nonumber\\
 &=\sup\left\{\|AT^nx\|_{\mathbf{R}(A^{1/2})}\,;\;x\in \overline{\mathcal{R}(A)},\,\|Ax\|_{\mathbf{R}(A^{1/2})}=1\right\}\nonumber\\
  &=\sup\left\{\|\widetilde{T^n}Ax\|_{\mathbf{R}(A^{1/2})}\,;\;x\in \overline{\mathcal{R}(A)},\,\|Ax\|_{\mathbf{R}(A^{1/2})}=1\right\}\nonumber\\
 & =\|\widetilde{T^n}\|_{\mathcal{B}(\mathbf{R}(A^{1/2}))},
\end{align}
for every $n\in \mathbb{N}^*$. Moreover, we shall prove that $\widetilde{T^n}=(\widetilde{T})^n$ for all $n\in \mathbb{N}^*$. Indeed, by induction we first show that $W_AT^n=(\widetilde{T})^nW_A$ for all  $n\in \mathbb{N}^*$. By Proposition \ref{prop_arias} the result holds for $n=1$. Assume that $W_AT^n=(\widetilde{T})^nW_A$ is true for some $n\in \mathbb{N}^*$. This implies that
\begin{align*}
W_AT^{n+1}
& =W_AT^{n}T \\
 &=(\widetilde{T})^nW_A T=(\widetilde{T})^{n+1}W_A.
\end{align*}
Further, since $\widetilde{T^n}$ is the unique operator which satisfies $W_AT^n=\widetilde{T^n}W_A$, then $\widetilde{T^n}=(\widetilde{T})^n$ for all  $n\in \mathbb{N}^*$. Hence, \eqref{compar01} yields
$$r_A(T)=r(\widetilde{T}),$$
where $r(\widetilde{T})$ denotes the classical spectral radius of $\widetilde{T}$ defined on $\mathcal{R}(A^{1/2})$.

It is known that  $r(\widetilde{T})\leq \omega (\widetilde{T})$, where $\omega(\cdot)$ denotes the classical numerical radius. Therefore, our proof will be complete if we show that $\omega_A(T)=\omega(\widetilde{T}).$

Notice that
\begin{align*}
\omega_A(T)
& =\sup\left\{|\langle Tx\mid x\rangle_A|\,;\;x\in \mathcal{H},\;\|x\|_A= 1\right\} \\
 &=\sup\left\{|( ATx, Ax)|\,;\;x\in \mathcal{H},\;\|Ax\|_{\mathbf{R}(A^{1/2})}= 1\right\}\\
  &=\sup\left\{|( \widetilde{T}Ax, Ax)|\,;\;x\in \mathcal{H},\;\|Ax\|_{\mathbf{R}(A^{1/2})}= 1\right\}.
\end{align*}
By using the decomposition $\mathcal{H}=\mathcal{N}(A^{1/2})\oplus \overline{\mathcal{R}(A^{1/2})}$, we get
$$\omega_A(T)=\sup\left\{|( \widetilde{T}Ax, Ax)|\,;\;x\in \overline{\mathcal{R}(A^{1/2})},\;\|Ax\|_{\mathbf{R}(A^{1/2})}= 1\right\}.$$
Moreover, the classical numerical radius of $\widetilde{T}$ on the Hilbert space $\mathbf{R}(A^{1/2})$ is given by
\begin{align*}
\omega(\widetilde{T})
& =\sup\left\{|( \widetilde{T}y, y)|\,;\;y\in \mathcal{R}(A^{1/2}),\;\|y\|_{\mathbf{R}(A^{1/2})}= 1\right\}\\
& =\sup\left\{|( \widetilde{T}A^{1/2}x, A^{1/2}x)|\,;\;x\in \mathcal{H},\;\|A^{1/2}x\|_{\mathbf{R}(A^{1/2})}= 1\right\}\\
& =\sup\left\{|( \widetilde{T}A^{1/2}x, A^{1/2}x)|\,;\;x\in \overline{\mathcal{R}(A^{1/2})},\;\|A^{1/2}x\|_{\mathbf{R}(A^{1/2})}= 1\right\}
\end{align*}
Since $A=A^{1/2}A^{1/2}$, it follows that $\omega_A(T)\leq \omega(\widetilde{T})$. To show the converse inequality (i.e., $\omega_A(T)\geq \omega(\widetilde{T})$), let
$$\alpha \in \left\{|( \widetilde{T}A^{1/2}x, A^{1/2}x)|\,;\;x\in \overline{\mathcal{R}(A^{1/2})},\;\|A^{1/2}x\|_{\mathbf{R}(A^{1/2})}= 1\right\}.$$
Then there exists $x\in \overline{\mathcal{R}(A^{1/2})}$ such that $\|A^{1/2}x\|_{\mathbf{R}(A^{1/2})}= 1$ and $\alpha=|( \widetilde{T}A^{1/2}x, A^{1/2}x)|$. In view of Lemma \ref{density01}, there exists a sequence $(x_n)_n\subset \mathcal{H}$ (which can chosen to be in $\overline{\mathcal{R}(A^{1/2})}$ because of the decomposition $\mathcal{H}=\mathcal{N}(A^{1/2})\oplus \overline{\mathcal{R}(A^{1/2})}$) such that $\displaystyle\lim_{n\to +\infty}\|Ax_n-A^{1/2}x\|_{\mathbf{R}(A^{1/2})}=0$. Hence we obtain
$$\alpha=\lim_{n\to +\infty}|( \widetilde{T}Ax_n, Ax_n)|\;\;\text{and}\;\;\lim_{n\to +\infty}\|Ax_n\|_{\mathbf{R}(A^{1/2})}= 1.$$
Let $y_n:=\frac{x_n}{\|Ax_n\|_{\mathbf{R}(A^{1/2})}}$, then clearly we deduce that
$$\alpha\in \overline{\left\{|( \widetilde{T}Ax, Ax)|\,;\;x\in \overline{\mathcal{R}(A^{1/2})},\;\|Ax\|_{\mathbf{R}(A^{1/2})}= 1\right\}}^{|\cdot|}.$$
This shows the desired inequality because $\sup \Lambda=\sup \overline{\Lambda}$ for any subset $\Lambda$ of $\mathbb{C}$. So, we deduce that $\omega_A(T)=\omega(\widetilde{T})$ and thus the proof is complete.
\end{proof}

By combining Theorem \ref{thm_spec_numer} together with Proposition \ref{anradius} we get the following corollary.
\begin{corollary}
If $T\in \mathcal{B}_{A^{1/2}}(\mathcal{H})$, then
\begin{equation}\label{radnumnorm}
r_A(T)\leq \omega_A(T)\leq \|T\|_A.
\end{equation}
\end{corollary}
Notice that the both inequalities in \eqref{radnumnorm} can simultaneously be strict even if $A=I$. Indeed, let $T=\begin{pmatrix} 1 & 1 \\ 0 & 1 \end{pmatrix}$ be an operator on $\mathbb{C}^2$. It can be checked that $r(T)=1$, $\|T\|=\frac{\sqrt{5}+1}{2}$ and $\omega(T)=\tfrac{3}{2}$.

Now, we introduce the following definition.
\begin{definition}
An operator $T\in \mathcal{B}_{A^{1/2}}(\mathcal{H})$ is said to be
\begin{itemize}
  \item [(i)] $A$-normaloid if $r_A(T)=\|T\|_A$.
  \item [(ii)] $A$-spectraloid if $r_A(T)=\omega_A(T)$.
\end{itemize}
\end{definition}

Notice that every $A$-normaloid operator is $A$-spectraloid as it is shown in the following theorem.
\begin{theorem}\label{theonormaloid}
Let $T\in \mathcal{B}_{A^{1/2}}(\mathcal{H})$ be an $A$-normaloid operator. Then,
$$r_A(T)=\omega_A(T)=\|T\|_A.$$
\end{theorem}
\begin{proof}
Follows immediately by using \eqref{radnumnorm}.
\end{proof}
\begin{remark}
It is important to note that an $A$-spectraloid operator is not necessarily $A$-normaloid even if $A=I$. Indeed, let $T=\begin{pmatrix} 1&0&0\\0&0&2\\0&0&0\end{pmatrix}$. It can be seen that $r(T)=1$, $\|T\|=2$ and $\omega(T)=1$. So, $r(T)=\omega(T)$ however, $r(T)\neq \|T\|$.
\end{remark}

The following proposition gives some characterizations of $A$-normaloid operators. Notice that characterizations of normaloid Hilbert space operators can be found in \cite{chanchan} and the references therein.
\begin{proposition}\label{charactenormailoid}
Let $T\in \mathcal{B}_{A^{1/2}}(\mathcal{H})$. Then, the following assertions are equivalent:
\begin{itemize}
  \item [(1)] $T$ is $A$-normaloid.
  \item [(2)] $\|T^n\|_A=\|T\|_A^n$ for all $n\in \mathbb{N}^*$.
  \item [(3)] $\omega_A(T)=\|T\|_A$.
    \item [(4)] There exists a sequence $(x_n)_n\subset\mathcal{H}$ such that $\|x_n\|_A=1$,
    \begin{equation*}
    \lim_{n\to \infty}\|T x_n\|_A=\|T\|_A \;\text{  and  }\; \lim_{n\to \infty}|\langle T x_n\mid x_n\rangle_A|=\omega_A(T).
    \end{equation*}
\end{itemize}
\end{proposition}
\begin{proof}
$(1)\Leftrightarrow(2):$ Assume that $T$ is $A$-normaloid, then $r_A(T)=\|T\|_A$. Since $T\in \mathcal{B}_{A^{1/2}}(\mathcal{H})$, then $\|T^n\|_A\leq\|T\|_A^n$ for all $n\in \mathbb{N}$. On the other hand, by using Proposition \ref{spectradiprop} and Corollary \ref{anradius} we obtain
$$\|T^n\|_A\geq r_A(T^n)=r_A(T)^n=\|T\|_A^n,$$
for all $n\in \mathbb{N}^*$. Conversely, we have $\|T\|_A=\|T^n\|_A^{1/n}\xrightarrow{n \to +\infty} r_A(T)$. So, $\|T\|_A=r_A(T)$.

$(1)\Leftrightarrow(3):$ According to the proof of Theorem \ref{thm_spec_numer}, we have $\omega_A(T)=\omega(\widetilde{T})$,  $\|T\|_A=\|\widetilde{T}\|_{\mathcal{B}(\mathbf{R}(A^{1/2}))}$ and $r_A(T)=r(\widetilde{T})$. So, $T\in \mathcal{B}_{A^{1/2}}(\mathcal{H})$ is $A$-normaloid if and only if $\widetilde{T}\in \mathcal{B}(\mathbf{R}(A^{1/2}))$ is a normaloid operator. On the other hand, since $\widetilde{T}\in \mathcal{B}(\mathbf{R}(A^{1/2}))$, then \cite[Proposition 6.27]{kub002} gives $r(\widetilde{T})=\|\widetilde{T}\|_{\mathcal{B}(\mathbf{R}(A^{1/2}))}$ if and only if $\omega(\widetilde{T})=\|\widetilde{T}\|_{\mathcal{B}(\mathbf{R}(A^{1/2}))}$. Therefore, the proof is complete.

$(1)\Leftrightarrow(4):$ Assume that there exists a sequence $(x_n)_n\subset\mathcal{H}$ such that $\|x_n\|_A=1$, $ \lim_{n\to \infty}\|T x_n\|_A=\|T\|_A$ and
$\lim_{n\to \infty}|\langle T x_n\mid x_n\rangle_A|=\omega_A(T)$. Set $y_n=Ax_n$. Then by using \eqref{usefuleq01}, we infer that $\|y_n\|_{\mathbf{R}(A^{1/2})}=\|Ax_n\|_{\mathbf{R}(A^{1/2})}=\|x_n\|_A=1$. Moreover, it can be seen that $\|T x_n\|_A=\|\widetilde{T} y_n\|_{\mathbf{R}(A^{1/2})}$ and $\langle T x_n\mid x_n\rangle_A=(\widetilde{T}y_n,y_n)$. So, we obtain
    \begin{equation*}
    \lim_{n\to \infty}\|\widetilde{T} y_n\|_{\mathbf{R}(A^{1/2})}=\|\widetilde{T}\|_{\mathcal{B}(\mathbf{R}(A^{1/2}))} \;\text{  and  }\; \lim_{n\to \infty}|(\widetilde{T} y_n\mid x_n)|=\omega(\widetilde{T}).
    \end{equation*}
This implies, by \cite[Theorem 1]{chanchan}, that $\widetilde{T}$ is a normaloid operator on $\mathbf{R}(A^{1/2})$. Hence $T$ is an $A$-normaloid operator.

Conversely, assume that $T$ is $A$-normaloid, then by assertion $(3)$ we see that $\omega_A(T)=\|T\|_A$. On the other hand, by definition of the $A$-numerical radius, there exists a sequence $(x_n)_n\subset\mathcal{H}$ such that $\|x_n\|_A=1$ and
$$\lim_{n\to \infty}|\langle T x_n\mid x_n\rangle_A|=\omega_A(T).$$
Moreover,
$$\|T\|_A\geq \|Tx_n\|_A\geq |\langle T x_n\mid x_n\rangle_A|\geq \omega_A(T)-\varepsilon=\|T\|_A-\varepsilon,$$
for every $\varepsilon> 0$ and $n$ large enough. Hence, $\displaystyle\lim_{n\to \infty}\|T x_n\|_A=\|T\|_A$.
\end{proof}
\begin{remark}\label{remuseful}
It follows from the proof of Proposition \ref{charactenormailoid} that if $T$ is an $A$-normaloid operator, then every $(x_n)_n\subset\mathcal{H}$ such that $\|x_n\|_A=1$ and $\displaystyle\lim_{n\to \infty}|\langle T x_n\mid x_n\rangle_A|=\omega_A(T)$ satisfies $\displaystyle\lim_{n\to \infty}\|T x_n\|_A=\|T\|_A$. However, it should be mentioned that not every $(x_n)_n\subset\mathcal{H}$ such that $\|x_n\|_A=1$ and $\displaystyle\lim_{n\to \infty}\|T x_n\|_A=\|T\|_A$ satisfies $\displaystyle\lim_{n\to \infty}|\langle T x_n\mid x_n\rangle_A|=\omega_A(T)$ even if $T$ is an $A$-normaloid operator (see \cite[p. 887]{chanchan}).
\end{remark}

A characterization of $A$-normaloid operators can be stated in terms of the $A$-maximal numerical range of operators which was introduced in \cite{bakfeki01} as follows.
\begin{definition}(\cite{bakfeki01})
Let $T\in \mathcal{B}^A(\mathcal{H})$. The $A$-maximal numerical range of $T$ is given by
$$W_{\max}^A(T)
 =\left\{\lambda\in \mathbb{C}\,;\exists\,(x_n)_n\subset \mathcal{H}\,;\,\|x_n\|_A=1,\displaystyle\lim_{n\to+\infty}\langle T x_n\mid x_n\rangle_A=\lambda\text{ and }\displaystyle\lim_{n\to+\infty}\|Tx_n\|_A=\|T\|_A\right\}.$$
\end{definition}

Similarly to the $A$-numerical radius of operators, we define the $A$-maximal numerical radius of operators as follows.
\begin{definition}
Let $T\in \mathcal{B}^A(\mathcal{H})$. The $A$-maximal numerical radius of $\mathbf{T}$ is given by
$$
\omega_{\max}^A(T)=\sup\{|\lambda|\,;\;\; \lambda\in W_{\max}^A(T)\,\}.
$$
\end{definition}
Now, we state the following characterization of $A$-normaloid operators.
\begin{proposition}
Let $T\in \mathcal{B}_{A^{1/2}}(\mathcal{H})$. Then, $T$ is $A$-normaloid if and only if $\omega_A(T)=\omega_{\max}^A(T)$.
\end{proposition}
\begin{proof}
Assume that $T$ is $A$-normaloid. Then, by Proposition \ref{charactenormailoid}, we have $\omega_A(T)=\|T\|_A$. We shall prove that $\omega_A(T)=\omega_{\max}^A(T)$. It can be observed that $W_{\max}^A(T)\subset \overline{W_A(T)}$, where $\overline{W_A(T)}$ is denoted to be the closure of the $A$-numerical range of $T$. Then, $\omega_{\max}^A(T)\leq\omega_A(T)$. On the other hand, by definition of $\omega_A(T)$, there exists a sequence $(x_n)_n\subset\mathcal{H}$ such that $\|x_n\|_A=1$ and $\displaystyle\lim_{n\to \infty}|\langle T x_n\mid x_n\rangle_A|=\omega_A(T)=\|T\|_A$. This yields, by Remark \ref{remuseful}, that $\displaystyle\lim_{n\to \infty}\|Tx_n\|_A=\|T\|_A$. Hence, $\omega_A(T)\in W_{\max}^A(T)$ and so $\omega_A(T)\leq \omega_{\max}^A(T)$. Therefore, $\omega_A(T)=\omega_{\max}^A(T)$.

Conversely, it is well known that $W_{\max}^A(T)$ is non-empty and a compact subset of $\mathbb{C}$ (see \cite{bakfeki01}). This implies that $\omega_{\max}^A(T)\in W_{\max}^A(T)$. Hence, there exists a sequence $(x_n)_n\subset\mathcal{H}$ such that $\|x_n\|_A=1$, $\displaystyle\lim_{n\to \infty}|\langle T x_n\mid x_n\rangle_A|=\omega_{\max}^A(T)=\omega_A(T)$ and $\displaystyle\lim_{n\to \infty}\|Tx_n\|_A=\|T\|_A$. Therefore, by Proposition \ref{charactenormailoid}, we deduce that $T$ is $A$-normaloid.
\end{proof}

Spectraloid operators are characterized by the equality $\omega(T^n)=\omega(T)^n$ for every natural number $n$ (see \cite{furutake}). Now, we aim to give here a similar characterization of $A$-spectraloid operators. To do this, we first state the following result.

\begin{theorem}
If $T\in \mathcal{B}_{A^{1/2}}(\mathcal{H})$, then
\begin{equation}\label{spetralnumerical01}
\omega_A(T^n)\leq \omega_A(T)^n,\;\forall\,n\in \mathbb{N}^*.
\end{equation}
\end{theorem}

\begin{proof}
Proceeding as in the proof of Theorem \ref{thm_spec_numer} and using \eqref{spetralclnumerical01} yields
\begin{align*}
\omega_A(T^n)
& = \omega(\widetilde{T^n})\\
 &=\omega(\widetilde{T}^n)\\
 &\leq \omega(\widetilde{T})^n= \omega_A(T)^n.
\end{align*}
So, \eqref{spetralnumerical01} is proved.
\end{proof}
Now, we are in a position to prove the next theorem.
\begin{theorem}
Let $T\in \mathcal{B}_{A^{1/2}}(\mathcal{H})$. Then, the following assertions are equivalent:
\begin{itemize}
  \item [(1)] $T$ is $A$-spectraloid.
  \item [(2)] $\omega_A(T^n)=\omega_A(T)^n$ for all $n\in \mathbb{N}^*$.
\end{itemize}
\end{theorem}
\begin{proof}
$(1)\Rightarrow(2):$ Assume that $T$ is $A$-spectraloid. In order to prove $(2)$ it suffices to show that $\omega_A(T^n)\geq\omega_A(T)^n$ for all $n\in \mathbb{N}^*$. By using Theorem \ref{thm_spec_numer} together with Proposition \ref{spectradiprop}, we see that
\begin{align*}
\omega_A(T^n)
& \geq r_A(T^n)\\
 &=\left[r_A(T)\right]^n=\omega_A(T)^n.
\end{align*}
$(1)\Rightarrow(2):$ By taking into consideration Theorem \ref{thm_spec_numer}, it suffices to prove that $r_A(T)\geq \omega_A(T)$. One has
$$\omega_A(T)=\omega_A(T^n)^{1/n}\leq \|T^n\|_A^{1/n},$$
for all $n\in \mathbb{N}^*$. This yields that $\omega_A(T)\leq r_A(T)$. Hence, the proof is complete.
\end{proof}

We close this section by refining the second inequality in \eqref{refine1}. Before that, it is useful to recall that F. Kittaneh proved in \cite{kittaneh01} that if $T\in \mathcal{B}(\mathcal{H})$, then
\begin{equation}\label{kitta01}
\omega(T)\leq \frac{1}{2}(\|T\|+\|T^2\|^{1/2}).
\end{equation}
This inequality has been used by Kittaneh in \cite{kittaneh01} in order to establish an estimate for the numerical radius of the Frobenius companion matrix. Now, we extend the inequality \eqref{kitta01} for the class of $A$-bounded operators as follows.
\begin{theorem}
Let $T\in \mathcal{B}_{A^{1/2}}(\mathcal{H})$. Then
\begin{equation}\label{kaisnew01}
\omega_A(T)\leq \frac{1}{2}(\|T\|_A+\|T^2\|_A^{1/2}).
\end{equation}
\end{theorem}
\begin{proof}
Proceeding as in the proof of Theorem \ref{thm_spec_numer} yields that $\omega_A(T)=\omega(\widetilde{T})$,  $\|T\|_A=\|\widetilde{T}\|_{\mathcal{B}(\mathbf{R}(A^{1/2}))}$ and $\|T^2\|_A=\|\widetilde{T}^2\|_{\mathcal{B}(\mathbf{R}(A^{1/2}))}$. So, we get \eqref{kaisnew01} by using \eqref{kitta01}
\end{proof}
To see that \eqref{kaisnew01} is a refinement of the second inequality in \eqref{refine1}, it suffices to use the fact that $\|T^2\|_A\leq \|T\|_A^2$ for every $T\in \mathcal{B}_{A^{1/2}}(\mathcal{H})$.

\begin{corollary}\label{corsharp}
Let $T\in \mathcal{B}_{A^{1/2}}(\mathcal{H})$ be such that $AT^2=0$. Then, $\omega_A(T)=\frac{1}{2}\|T\|_A$.
\end{corollary}
\begin{proof}
Notice first that since $AT^2=0$, then $\|T^2\|_A=0$. Further, by combining \eqref{refine1} together with \eqref{kaisnew01}, we get
$$\frac{1}{2}\|T\|_A\leq \omega_A(T) \leq \frac{1}{2}(\|T\|_A+\|T^2\|_A^{1/2}),$$
for every $T\in \mathcal{B}_{A^{1/2}}(\mathcal{H})$. Therefore, if $\|T^2\|_A=0$ then $\omega_A(T)=\frac{1}{2}\|T\|_A$ as desired.
\end{proof}
\begin{remark}
By using Theorem \ref{theonormaloid} together with Corollary \ref{corsharp}, one observes that the inequalities in \eqref{refine1} are sharp.
\end{remark}
\section{Some examples of $A$-normaloid operators}
It is well know that we have the following proper inclusions between the following classes of Hilbert space operators (see \cite{furuta0})
\begin{equation}\label{inclusionsproper}
\text{Self-adjoint}\subsetneq \text{normal}\subsetneq\text{hyponormal}\subsetneq\text{paranormal}\subsetneq\text{normaloid}.
\end{equation}
Recall from \cite{furuta0} that an operator $T\in \mathcal{B}_{A^{1/2}}(\mathcal{H})$ is said to be paranormal if $\|Tx\|^2\leq \|T^2x\|$ for all $x\in \mathcal{H}$ such that $\|x\|=1$.

Our aim in this section is to give some examples of $A$-normaloid operators and to study the relationship between them. Mainly, we introduce the class of $A$-paranormal operators and we will show that the first inclusion in \eqref{inclusionsproper} fails to hold in the case of semi-Hilbertian space operators.

Recently, the class of $A$-hyponormal operators was introduced by O.A.M. Sid Ahmed et al. in \cite{sidb1}. Their definition extends the classical definition of hyponormal Hilbert space operators and reads as follows.
\begin{definition}
An operator $T\in\mathcal{B}_A(\mathcal{H})$ is said to be $A$-hyponormal if
$$[T^\sharp,T]:=T^{\sharp}T-TT^{\sharp}\geq_A 0,$$
or equivalently if $\|Tx\|_A\geq \|T^\sharp x\|_A$ for all $x\in \mathcal{H}$.
\end{definition}
If an operator $T\in\mathcal{B}_A(\mathcal{H})$ satisfies $[T^\sharp,T]=0$, then $T$ is said to be $A$-normal. For more details on these classes of operators, we refer the readers to \cite{sidb1,saddi,acg3} and the references therein. Before we move on, let us emphasize the following remark.
\begin{remark}
If $T$ is an hyponormal operator on a finite-dimensional Hilbert space $\mathcal{H}$, then $T$ is necessarily a normal operator. Indeed, since
 $T^*T - TT^* \ge 0$, then $\sigma(A^*A - AA^*) \subseteq [0, +\infty[$. Moreover, clearly we have $\operatorname{Tr}(T^*T-TT^*) = 0$. This implies that $\sigma(T^*T-TT^*) = \{0\}$ which in turn yields $T^*T-TT^* = 0$. However, an $A$-hyponormal operator is in general not $A$-normal even if $\mathcal{H}$ is finite-dimensional as it is shown in the following example.
\end{remark}
\begin{example}
Let $A=\begin{pmatrix}1&1\\1&1\end{pmatrix}$ and $T=\begin{pmatrix}2&2\\0&0\end{pmatrix}$ be operators acting on $\mathbb{C}^2$. One can verify that $T\in \mathcal{B}_A(\mathcal{H})$ and $T^\sharp=\begin{pmatrix}1&1\\1&1\end{pmatrix}$. Moreover, it can be seen that $[T^\sharp,T]\neq0$ and $A[T^\sharp,T]=0$. Hence, $T$ is $A$-hyponormal and not $A$-normal.
\end{example}

For the sequel, we set
$$S^A(0,1):=\{x \in \mathcal{H}\,;\;\|x\|_A=1\}.$$
In the following definition, we introduce the class of $A$-paranormal operators.
\begin{definition}
An operator $T\in \mathcal{B}_{A^{1/2}}(\mathcal{H})$ is said to be $A$-paranormal if
$$\|Tx\|_A^2\leq \|T^2x\|_A,\;\;\forall\,x\in S^A(0,1).$$
\end{definition}
\begin{lemma}
Let $T\in \mathcal{B}_A(\mathcal{H})$ be an $A$-hyponormal operator. Then, $T$ is $A$-paranormal.
\end{lemma}
\begin{proof}
Let $x \in S^A(0,1)$. Then, by using the fact $T$ is $A$-hyponormal, we see that
\begin{align*}
\|Tx\|^2
& = \langle Tx\mid Tx\rangle_A\\
 &=\langle T^\sharp Tx\mid x\rangle_A\\
  &\leq\|T^\sharp Tx\|_A\\
   &\leq\|T^2x\|_A.
\end{align*}
Hence, $T$ is $A$-paranormal as required.
\end{proof}

Now, we state our main result in this section.
\begin{theorem}\label{aparanormal01}
Let $T\in \mathcal{B}_{A^{1/2}}(\mathcal{H})$ be an $A$-paranormal operator. Then, $T$ is $A$-normaloid. Moreover, we have
\begin{equation}\label{aparequa}
r_A(T)=\omega_A(T)=\|T\|_A.
\end{equation}
\end{theorem}
\begin{proof}
We shall prove by induction that
\begin{equation}\label{ineq014}
\|T^nx\|_A\geq \|Tx\|_A^n,\;\,\forall\, n\in \mathbb{N}^*,\;x\in S^A(0,1).
\end{equation}
 If $n=1$, then the property \eqref{ineq014} holds trivially. Let $n\in  \mathbb{N}^*$. Assume that $\|T^nx\|_A\geq \|Tx\|_A^n$ for all $x\in S^A(0,1)$ and we will show that $\|T^{n+1}x\|_A\geq \|Tx\|_A^{n+1}$ for all $x\in S^A(0,1)$. Let $x\in S^A(0,1)$ be such that $\|Tx\|_A\neq0$. Then, one can see that
 \begin{align*}
\|T^{n+1}x\|_A
& =\|T^n(\tfrac{Tx}{\|Tx\|_A})\|_A\|Tx\|_A \\
 &\geq \|T(\tfrac{Tx}{\|Tx\|_A})\|_A^n\|Tx\|_A\\
 &=\|T^2x\|_A^n \|Tx\|_A^{1-n}\\
 &\geq \|Tx\|_A^{2n} \|Tx\|_A^{1-n}=\|Tx\|_A^{n+1}.
\end{align*}
 Notice that if $x\in S^A(0,1)$ and satisfies $\|Tx\|_A=0$, then clearly $\|T^{n+1}x\|_A\geq\|Tx\|_A^{n+1}$. So, $\|T^{n+1}x\|_A\geq\|Tx\|_A^{n+1}$ for all $x\in S^A(0,1)$. Hence, \eqref{ineq014} is proved by the mathematical induction. Therefore, by taking the supremum over all $x\in S^A(0,1)$ in \eqref{ineq014}, we get $\|T^n\|_A\geq \|T\|_A^n$, for all $n\in \mathbb{N}^*$. On the other hand, clearly we have $\|T^n\|_A\leq\|T\|_A^n$, for all $n\in \mathbb{N}^*$. Thus, $\|T^n\|_A= \|T\|_A^n$, for all $n\in \mathbb{N}^*$. So, by applying Proposition \ref{charactenormailoid} we deduce that $T$ is $A$-normaloid. Moreover, \eqref{aparequa} follows immediately by using Theorem \ref{theonormaloid}.
\end{proof}

\begin{remark}
Since the $A$-normality implies the $A$-hyponormality, and an $A$-hyponormal operator is $A$-paranormal, then by taking into account Theorem \ref{aparanormal01} we get the following inclusions:
\begin{equation}\label{newinclusion}
\text{$A$-normal}\subsetneq\text{$A$-hyponormal}\subsetneq\text{$A$-paranormal}\subsetneq\text{$A$-normaloid}.
\end{equation}
Notice that it was shown in \cite{saddi} that if $T$ is $A$-normal, then $r_A(T)=\|T\|_A$. Further, if $T$ is $A$-normal, then $r_A(T)=\omega_A(T)$ (see \cite{fg}). On the other hand, by using \eqref{newinclusion} together with Theorem \ref{aparanormal01}, we obtain that $r_A(T)=\omega_A(T)=\|T\|_A$ for every $A$-normal operator $T$. So, an alternative proof of the above two results established in \cite{saddi,fg} is given.
\end{remark}

It is well known that every self adjoint operator is normal. However, an $A$-self adjoint operator is not in general neither $A$-normal nor $A$-hyponormal as it is shown in the following example.

\begin{example}
Let $A=\begin{pmatrix}2&0&2\\0&1&0\\2&0&2\end{pmatrix}$ and $T=\begin{pmatrix}1&0&1\\0&0&0\\0&0&0\end{pmatrix}$ be operators acting on $\mathbb{C}^3$. Observe that if $A$ is positive and $T$ is A-selfadjoint. A short calculation shows that $$T^\sharp=\begin{pmatrix}\frac{1}{2}&0&0\\0&0&0\\ \frac{1}{2}&0&0\end{pmatrix},\;T^\sharp T=\begin{pmatrix}\frac{1}{2}&0&\frac{1}{2}\\0&0&0\\\frac{1}{2}&0&\frac{1}{2}\end{pmatrix}\;\text{and}\;\; TT^\sharp= \begin{pmatrix}1&0&0\\0&0&0\\0&0&0\end{pmatrix}.$$
So, clearly $T$ is not $A$-normal. Moreover, it is not difficult to see that $T$ is also not $A$-hyponormal. However, an $A$-self-adjoint operator is $A$-paranormal as it it shown in the next proposition.
\end{example}

\begin{proposition}
Let $T \in \mathcal{B}(\mathcal{H})$ be an $A$-self-adjoint operator. Then $T$ is $A$-paranormal.
\end{proposition}
\begin{proof}
Let $x \in S^A(0,1)$. Then, by using the fact $T$ is $A$-self-adjoint and the Cauchy-Schwarz inequality, we see that
\begin{align*}
\|Tx\|^2
& = \langle ATx\mid Tx\rangle\\
 &=\langle T^*Ax\mid Tx\rangle\\
 &=\langle x\mid T^2x\rangle_A\\
   &\leq\|T^2x\|_A.
\end{align*}
\end{proof}
The following corollary is an immediate consequence of Proposition \ref{charactenormailoid}. and Theorem \ref{aparanormal01}.

\begin{corollary}\label{aself3}
Let $T \in \mathcal{B}(\mathcal{H})$ be an $A$-self-adjoint operator. Then, $T$ is $A$-normaloid and so it satisfies
$$r_A(T)=\|T\|_A=\omega_A(T).$$
\end{corollary}

Notice that there are two other ways to prove that an $A$-self-adjoint operator is $A$-normaloid. Firstly, in view of \cite[Lemma 2.1.]{zamani1}, we have $\omega_A(T)=\|T\|_A$. Thus $T$ is $A$-normaloid by applying Proposition \ref{charactenormailoid}. Secondly, it was shown in \cite[Theorem 5.1.]{bakfeki04}, that if $T$ is an $A$-self-adjoint operator then, $\|T^n\|_A=\|T\|_A^n$ for all $n\in \mathbb{N}^*$. So, we get the desired property by using Proposition \ref{charactenormailoid}.


\end{document}